\documentclass[reqno]{amsart}
\usepackage{amsmath, amsfonts, amsthm, amssymb, setspace, textcomp, bbm, multirow}
\usepackage{geometry}
\newtheorem{theorem}{Theorem}[section]

\newtheorem{proposition}[theorem]{Proposition}

\theoremstyle{definition}

\theoremstyle{remark}

\numberwithin{equation}{section}

\newcommand{\Parans}[1]{\left(#1\right)}

\newcommand{\SBrackets}[1]{\left[#1\right]}

\newcommand{\aqprod}[3]{\Parans{#1;#2}_{#3}}
\newcommand{\jacprod}[2]{\SBrackets{#1;#2}_{\infty}}

\author{CHRIS JENNINGS-SHAFFER}
\address{Department of Mathematics, University of Florida\\
Gainesville, Florida 32611, USA
\endgraf cjenningsshaffer@ufl.edu}



\title{Ranks For Two Partition Quadruple Functions}

\allowdisplaybreaks
\begin{document}

\allowdisplaybreaks

\begin{abstract}
Recently the author introduced two new integer partition quadruple functions, 
which satisfy Ramanujan-type congruences modulo $3$, $5$, $7$, and $13$. Here
we reprove the congruences modulo $3$, $5$, and $7$ by defining a rank-type
statistic that gives a combinatorial refinement of the congruences.
\end{abstract}

\maketitle

\section{Introduction and Statement of Results}
\allowdisplaybreaks

In \cite{JS1} the author introduced the two partition quadruple functions
$u(n)$ and $v(n)$. We recall a partition of a positive integer $n$ is a non-increasing
sequence of positive integers that sum to $n$. For example the partitions of
$5$ are exactly $5$, $4+1$, $3+2$, $3+1+1$, $2+2+1$, $2+1+1+1$, and $1+1+1+1+1$.
We say a quadruple $(\pi_1,\pi_2,\pi_3,\pi_4)$ of partitions
is a partition quadruple of $n$ if altogether the parts of $\pi_1$, $\pi_2$,
$\pi_3$, and $\pi_4$ sum to $n$.
For a partition $\pi$, we let $s(\pi)$ denote the smallest part of $\pi$
and $\ell(\pi)$ denote the largest part of $\pi$. We use the conventions
that the empty partition has smallest part $\infty$ and largest part $0$. We 
let $u(n)$ denote the number of partition quadruples 
$(\pi_1,\pi_2,\pi_3,\pi_4)$ of $n$ such that $\pi_1$ is non-empty,
$s(\pi_1)\le s(\pi_2)$, $s(\pi_1)\le s(\pi_3)$, $s(\pi_1)\le s(\pi_4)$,
and $\ell(\pi_4)\le 2s(\pi_1)$.   
We let $v(n)$ denote the number of partition quadruples 
$(\pi_1,\pi_2,\pi_3,\pi_4)$ counted by $u(n)$, where additionally
the smallest part of $\pi_1$ 
appears at least twice.   
As example, $u(3)=15$ as the relevant partition quadruples are
$(3,-,-,-)$,
$(2+1,-,-,-)$,
$(1+1+1,-,-,-)$,
$(1+1,1,-,-)$,
$(1+1,-,1,-)$,
$(1+1,-,-,1)$,
$(1,2,-,-)$,
$(1,-,2,-)$,
$(1,-,-,2)$,
$(1,1+1,-,-)$,
$(1,-,1+1,-)$,
$(1,-,-,1+1)$,
$(1,1,1,-)$,
$(1,1,-,1)$, and
$(1,-,1,1)$.
Also from this list we see that $v(3)=4$.
In \cite{JS1} we proved the following congruences for $u(n)$ and $v(n)$.
\begin{theorem}\label{TheoremCongruences}
\begin{align*}
	u(3n) \equiv 0 \pmod{3}
	,\\	
	u(5n) \equiv 0 \pmod{5}
	,\\	
	u(5n+3) \equiv 0 \pmod{5}
	,\\	
	u(7n) \equiv 0 \pmod{7}
	,\\	
	u(7n+5) \equiv 0 \pmod{7}
	,\\	
	u(13n) \equiv 0 \pmod{13}
	,\\	
	v(3n+1) \equiv 0 \pmod{3}
	,\\	
	v(5n+1) \equiv 0 \pmod{5}
	,\\	
	v(5n+4) \equiv 0 \pmod{5}
	,\\	
	v(13n+10) \equiv 0 \pmod{13}
.
\end{align*}
\end{theorem}
The proof was to use various identities between products and generalized 
Lambert series to determine formulas modulo $\ell$ for the 
$\ell$-dissections of 
generating functions
for $u(n)$ and $v(n)$, with the appropriate terms of the dissections being zero.

We use the standard product notation,
\begin{align*}
	\aqprod{z}{q}{n} 
		&= \prod_{j=0}^{n-1} (1-zq^j)
	,
	&\aqprod{z}{q}{\infty} 
		&= \prod_{j=0}^\infty (1-zq^j)
	,\\
	\aqprod{z_1,\dots,z_k}{q}{n} 
		&= \aqprod{z_1}{q}{n}\dots\aqprod{z_k}{q}{n}
	,
	&\aqprod{z_1,\dots,z_k}{q}{\infty} 
		&= \aqprod{z_1}{q}{\infty}\dots\aqprod{z_k}{q}{\infty},
	\\
	\jacprod{z}{q} 
		&= \aqprod{z,q/z}{q}{\infty}
	,	
	&\jacprod{z_1,\dots,z_k}{q} &= \jacprod{z_1}{q}\dots\jacprod{z_k}{q}
.
\end{align*}
By summing according to $n$ being the smallest part of the
partition $\pi_1$, we find that generating functions for $u(n)$
and $v(n)$ are given by
\begin{align*}
	U(q)
	&=
	\sum_{n=0}^\infty u(n)q^n
	=
	\sum_{n=1}^\infty 
	\frac{q^n}{\aqprod{q^n}{q}{\infty}\aqprod{q^n}{q}
		{\infty}\aqprod{q^n}{q}{\infty}\aqprod{q^n}{q}{n+1}}
	\\
	&=
		q+5q^2+15q^3+44q^4+105q^5+252q^6+539q^7+1135q^8+2259q^9+4390q^{10}+\dots
	,\\
	V(q)
	&=
	\sum_{n=0}^\infty v(n)q^n
	=
	\sum_{n=1}^\infty 
	\frac{q^{2n}}{\aqprod{q^n}{q}{\infty}\aqprod{q^n}{q}
		{\infty}\aqprod{q^n}{q}{\infty}\aqprod{q^n}{q}{n+1}}
	\\
	&=
		q^2+4q^3+15q^4+39q^5+105q^6+237q^7+530q^8+1100q^9+2223q^{10}+\dots
.
\end{align*}

Recently in private correspondence Garvan conjectured that one could use the series
\begin{align*}
	\sum_{n=1}^\infty 
	\frac{q^n\aqprod{q^n,q^n,q^{2n+1}}{q}{\infty}}
	{\aqprod{q^{3n}}{q^3}{\infty}^2}
	,\hspace{20pt}
	\sum_{n=1}^\infty 
		\frac{q^n\aqprod{q^n,q^{2n+1}}{q}{\infty}}
		{\aqprod{q^{5n}}{q^5}{\infty}}
	,\hspace{20pt}
	\sum_{n=1}^\infty 
		\frac{q^n\aqprod{q^n,\zeta_7^3q^n, \zeta_7^4q^n,q^{2n+1}}{q}{\infty}}
		{\aqprod{q^{7n}}{q^7}{\infty}},
\end{align*}
where $\zeta_\ell$ is a primitive $\ell^{th}$ root of unity,
as rank functions to prove the congruences $u(3n)\equiv 0 \pmod{3}$, 
$u(5n) \equiv u(5n+3) \equiv 0 \pmod{5}$, and 
$u(7n) \equiv u(7n+5) \equiv 0 \pmod{7}$ respectively. These functions correspond to the
$z=\zeta_3$, $\zeta_5$, and $\zeta_7$ cases of
$F(z^2,z^{-2},z;q)$, where $F(\rho_1,\rho_2,z;q)$ is a function the author
studied in \cite{JS5} defined by
\begin{align*}
	F(\rho_1,\rho_2,z;q)
	&=
	\frac{\aqprod{q}{q}{\infty}}{\aqprod{z,z^{-1},\rho_1,\rho_2}{q}{\infty}}
	\sum_{n=1}^\infty
	\frac{\aqprod{z,z^{-1},\rho_1,\rho_2}{q}{n}(\tfrac{q}{\rho_2\rho_2})^n}{\aqprod{q}{q}{2n}}	
	.
\end{align*}
In this article we give the proof of this as well as give the corresponding
rank function for $V(q)$.
We let $RU(z,q)=F(z^2,z^{-2},z;q)$ and $RV(z,q)=G(z^2,z^{-2},z;q)$,
where
\begin{align*}
	G(\rho_1,\rho_2,z;q)
	&=
	\frac{\aqprod{q}{q}{\infty}}{\aqprod{z,z^{-1},\rho_1,\rho_2}{q}{\infty}}
	\sum_{n=1}^\infty
	\frac{\aqprod{z,z^{-1},\rho_1,\rho_2}{q}{n}(\tfrac{q^2}{\rho_2\rho_2})^n}{\aqprod{q}{q}{2n}}	
	.
\end{align*}
We prove the following identities for these functions.
\begin{theorem}\label{TheoremMain}
Let $\zeta_\ell$ denote a primitive $\ell^{th}$ root of unity. Then
\begin{align}
	\label{TheoremMainRU3}
	RU(\zeta_3,q)
	&=
		\frac{q^7}{\aqprod{q^3}{q^3}{\infty}}
		\sum_{n=-\infty}^\infty \frac{ (-1)^{n}q^{\frac{9n^2+27n}{2}}}{(1-q^{9n+6}) }	
		-		
		\frac{q^5}{\aqprod{q^3}{q^3}{\infty}}
		\sum_{n=-\infty}^\infty \frac{ (-1)^{n}q^{\frac{9n^2+21n}{2}}}{(1-q^{9n+6}) }	
	,\\
	\label{TheoremMainRV3}
	RV(\zeta_3,q)
	&=
		-\frac{q^3}{\aqprod{q^3}{q^3}{\infty}}
		\sum_{n=-\infty}^\infty \frac{ (-1)^{n}q^{\frac{9n^2+15n}{2}}}{(1-q^{9n+6}) }	
		+		
		\frac{q^5}{\aqprod{q^3}{q^3}{\infty}}
		\sum_{n=-\infty}^\infty \frac{ (-1)^{n}q^{\frac{9n^2+21n}{2}}}{(1-q^{9n+6}) }	
	,\\
	\label{TheoremMainRU5}
	RU(\zeta_5,q)
	&=
		q\frac{\aqprod{q^{25}}{q^{25}}{\infty}}{\jacprod{q^5}{q^{25}}^2}							
		-
		\frac{q^7}{\aqprod{q^{25}}{q^{25}}{\infty}\jacprod{q^{10}}{q^{25}}}
		\sum_{n=-\infty}^\infty \frac{(-1)^nq^{\frac{25n^2+45n}{2}}  }{1-q^{25n+10}}
		\nonumber\\&		
		-
		\frac{q^4}{\aqprod{q^{25}}{q^{25}}{\infty}\jacprod{q^5}{q^{25}}} 
		\sum_{n=-\infty}^\infty \frac{(-1)^nq^{\frac{25n^2+35n}{2}}  }{1-q^{25n+10}}
	,\\
	\label{TheoremMainRV5}
	RV(\zeta_5,q)
	&=
		-\frac{q^5}{\aqprod{q^{25}}{q^{25}}{\infty}\jacprod{q^5}{q^{25}}}	
		\sum_{n=-\infty}^\infty \frac{(-1)^nq^{\frac{25n^2+35n}{2}}  }{1-q^{25n+15}}
		+
		\frac{q^{12}}{\aqprod{q^{25}}{q^{25}}{\infty}\jacprod{q^{10}}{q^{25}}}
		\sum_{n=-\infty}^\infty \frac{(-1)^nq^{\frac{25n^2+55n}{2}}  }{1-q^{25n+15}}
		\nonumber\\&\quad		
		+
		q^2\frac{\aqprod{q^{25}}{q^{25}}{\infty}}{\jacprod{q^5,q^{10}}{q^{25}}}		
		-
		q^3\frac{\aqprod{q^{25}}{q^{25}}{\infty}}{\jacprod{q^{10}}{q^{25}}^2}
	,\\
	\label{TheoremMainRU7}	
	RU(\zeta_7,q)
	&=
  		q\frac{\aqprod{q^{49}}{q^{49}}{\infty}\jacprod{q^{21}}{q^{49}}}{\jacprod{q^{7}}{q^{49}} \jacprod{q^{14}}{q^{49}}^2}
		-
		(\zeta_7^2+\zeta_7^5)
		\frac{q^{15}}{\aqprod{q^{49}}{q^{49}}{\infty} \jacprod{q^{21}}{q^{49}}  }
		\sum_{n=-\infty}^\infty\frac{(-1)^n q^{\frac{49n^2+91n}{2}}}{1-q^{49n+21}}		
		\nonumber\\&\quad
		-
		(\zeta_7^3+\zeta_7^4)q^2 
		\frac{\aqprod{q^{49}}{q^{49}}{\infty}}{\jacprod{q^{7},q^{14}}{q^{49}}}
		+
		q^3\frac{\aqprod{q^{49}}{q^{49}}{\infty}}{\jacprod{q^{7},q^{21}}{q^{49}}}		
		+
		(\zeta_7+\zeta_7^6)
		q^4\frac{\aqprod{q^{49}}{q^{49}}{\infty}}{\jacprod{q^{14}}{q^{49}}^2}
		\nonumber\\&\quad		
		+
		(\zeta_7+\zeta_7^6)
		\frac{q^{11}}{\aqprod{q^{49}}{q^{49}}{\infty} \jacprod{q^{14}}{q^{49}} }
		\sum_{n=-\infty}^\infty\frac{(-1)^n q^{\frac{49n^2+77n}{2}}}{1-q^{49n+21}}				
		+
		(1+\zeta_7^3+\zeta_7^4)
		q^6\frac{\aqprod{q^{49}}{q^{49}}{\infty}}{\jacprod{q^{21}}{q^{49}}^2}
		\nonumber\\&\quad
		-
		(1+\zeta_7^3+\zeta_7^4)
		\frac{q^6 }{\aqprod{q^{49}}{q^{49}}{\infty} \jacprod{q^{7}}{q^{49}}}
		\sum_{n=-\infty}^\infty\frac{(-1)^n q^{\frac{49n^2+63n}{2}}}{1-q^{49n+21}}				
	.	
\end{align}
\end{theorem}
The congruences modulo $3$, $5$, and $7$ of Theorem \ref{TheoremCongruences} 
are a corollary to Theorem \ref{TheoremMain} by the standard argument of ranks and
cranks. We let
\begin{align*}
	RU(z,q) &= \sum_{n=1}^\infty\sum_{m=-\infty}^\infty ru(m,n)z^mq^n
	,
	&RV(z,q) &= \sum_{n=1}^\infty\sum_{m=-\infty}^\infty rv(m,n)z^mq^n
	,\\
	ru(k,\ell,n) &= \sum_{m\equiv k\pmod{\ell}} ru(m,n)
	,
	&rv(k,\ell,n) &= \sum_{m\equiv k\pmod{\ell}} rv(m,n)
	.\\
\end{align*}
Since $U(q)=RU(1,q)$ and $V(q)=RV(1,q)$ we have that
\begin{align*}
	u(n) &= \sum_{k=0}^{\ell-1} ru(k,\ell,n)
	,
	&v(n) &= \sum_{k=0}^{\ell-1} rv(k,\ell,n)
	.
\end{align*}
Also we see that
\begin{align*}
	RU(\zeta_\ell,q)
	&=
		\sum_{n=1}^\infty \left(\sum_{k=0}^{\ell-1} ru(k,\ell,n)\zeta_\ell^k\right) q^n
	,
	&RV(\zeta_\ell,q)
	&=
		\sum_{n=1}^\infty \left(\sum_{k=0}^{\ell-1} rv(k,\ell,n)\zeta_\ell^k\right) q^n
	.
\end{align*}

If $\ell$ is prime and the coefficient of $q^N$ in $RU(\zeta_\ell,q)$ is zero, then
\begin{align*}
	ru(0,\ell,N)+ru(1,\ell,N)\zeta_\ell+\dots+ru(\ell-1,\ell,N)\zeta_\ell^{\ell-1} = 0
,
\end{align*}
and so
\begin{align*}
	ru(0,\ell,N)=ru(1,\ell,N)=\dots=ru(\ell-1,\ell,N),
\end{align*}
as the minimal polynomial for $\zeta_\ell$ is $1+x+\dots+x^{\ell-1}$.
Thus $ru(N)=\ell\cdot ru(k,\ell,N)\equiv 0\pmod{\ell}$. By Theorem 
\ref{TheoremMain} the coefficients of
$q^{3n}$ in $RU(\zeta_3,q)$; $q^{5n}$ and $q^{5n+3}$ in $RU(\zeta_5,q)$;
and $q^{7n}$ and $q^{7n+5}$ in $RU(\zeta_7,q)$ are all zero which yields the
congruences for $u(n)$. The explanation for $rv(n)$ is similar.

It is worth noting that the $ru(k,\ell,N)$ being equal 
and the $rv(k,\ell,N)$ being equal is a stronger result than the congruences alone. 
Also the $ru(m,n)$ and $rv(m,n)$ give statistics on the partition
quadruples counted by $u(n)$ and $v(n)$ that yield a combinatorial refinement of the
congruences. For this we recall that $s(\pi)$ is the smallest part of $\pi$ and 
let $\#(\pi)$ denote the number of a parts of $\pi$. 
For a quadruple $(\pi_1,\pi_2,\pi_3,\pi_4)$ we let 
$\omega(\pi_1,\pi_2,\pi_3,\pi_4)$ denote the number of parts of $\pi_1$ that are either
$s(\pi_1)$ or are larger than $s(\pi_1)+\#(\pi_4)$. We note that if
$\#(\pi_4)=0$, then $\omega(\pi)=\#(\pi_1)$.
We define the $u$-rank and $v$-rank of $(\pi_1,\pi_2,\pi_3,\pi_4)$ by
\begin{align*}
	u\mbox{-rank}(\pi_1,\pi_2,\pi_3,\pi_4)
	&=		
	\omega(\pi_1,\pi_2,\pi_3,\pi_4)-1+2\#(\pi_2)-2\#(\pi_3)-\#(\pi_4)
	,\\
	v\mbox{-rank}(\pi_1,\pi_2,\pi_3,\pi_4)
	&=		
	\omega(\pi_1,\pi_2,\pi_3,\pi_4)-2+2\#(\pi_2)-2\#(\pi_3)-\#(\pi_4)
.
\end{align*}
We will prove the following.
\begin{theorem}\label{TheoremCombinatorics}
Let $U$ denote the set of partitions quadruples counted by $u(n)$, 
that is to say,
\begin{align*}
	U = \{(\pi_1,\pi_2,\pi_3,\pi_4): \ell(\pi_4)\le 2s(\pi_1)<\infty \mbox{ and } s(\pi_1)\le s(\pi_i) \mbox{ for } i=2,3,4    \}
.
\end{align*}
In the same fashion let $V$ denote the set of partitions quadruples counted by $v(n)$.
Then $ru(m,n)$ is the number of partition quadruples from $U$ of $n$ with $u$-rank equal to $m$
and $rv(m,n)$ is the number of partition quadruples from $V$ of $n$ with $v$-rank equal to $m$.
Furthermore,
\begin{enumerate}
\item[(i)]
the residue of the $u$-rank mod $3$ divides the partition quadruples from $U$
of $3n$ into $3$ equal classes,
\item[(ii)]
the residue of the $v$-rank mod $3$ divides the partition quadruples from $V$
of $3n+1$ into $3$ equal classes,
\item[(iii)]
the residue of the $u$-rank mod $5$ divides the partition quadruples from $U$
of $5n$ and of $5n+3$ into $5$ equal classes,
\item[(iv)]
the residue of the $v$-rank mod $5$ divides the partition quadruples from $V$
of $5n+1$ and of $5n+4$ into $5$ equal classes, and
\item[(v)]
the residue of the $u$-rank mod $7$ divides the partition quadruples from $U$
of $7n$ and of $7n+5$ into $7$ equal classes.
\end{enumerate}
\end{theorem}

As example of Theorem 1.3, we have $u(3)\equiv 0\pmod{3}$ and 
$u(3)\equiv 0\pmod{5}$ by considering the following table of values:
\begin{align*}
\begin{array}{cccc|c|c|c|c}
	\pi_1 & \pi_2 & \pi_3 & \pi_4 & \omega & u\mbox{-rank} & u\mbox{-rank} \pmod{3} & u\mbox{-rank} \pmod{5}
	\\
	\hline
	3&-&-&-&1&0&0&0	
	\\
	\hline
	2+1&-&-&-&2&1&1&1
	\\
	\hline
	1+1+1&-&-&-&3&2&2&2
	\\
	\hline
	1+1&1&-&-&2&3&0&3
	\\
	\hline
	1+1&-&1&-&2&-1&2&4
	\\
	\hline
	1+1&-&-&1&2&0&0&0
	\\
	\hline
	1&2&-&-&1&2&2&2
	\\
	\hline
	1&-&2&-&1&-2&1&3
	\\
	\hline
	1&-&-&2&1&-1&2&4
	\\
	\hline
	1&1+1&-&-&1&4&1&4
	\\
	\hline
	1&-&1+1&-&1&-4&2&1
	\\
	\hline
	1&-&-&1+1&1&-2&1&3
	\\
	\hline
	1&1&1&-&1&0&0&0
	\\
	\hline
	1&1&-&1&1&1&1&1
	\\
	\hline
	1&-&1&1&1&-3&0&2	
\end{array}
\end{align*}

Two remarks are in order. The first is that Theorem 1.3 shows that we are doing
something more than reproving some of the congruences for $u(n)$ and $v(n)$.
Previously we simply knew the congruences held, whereas now we have a much 
stronger refinement. This refinement fits into the rich framework of
partition rank and cranks. To review the history of this subject, one should
consult the works in \cite{AndrewsBerndt3, AndrewsGarvan, AS, Dyson, Garvan1}.
The second is that the proof of Theorem 1.2 is considerably easier than the 
original proofs of the congruences. One possible explanation for this is that
it is not entirely clear what kind of functions are $U(q)$ and $V(q)$;
are they modular, mock modular, or quasi-mock modular? However given
a form of $RU(z,q)$ and $RV(z,q)$ in terms of generalized Lambert series,
it clear by the works of Zwegers \cite{Zwegers} that for $z$ a root of unity,
other than $\pm 1$, they are mock modular forms.

In Section 2 we give a few preliminary identities necessary for the proof of
Theorem \ref{TheoremMain}, in Section 3 we prove Theorem \ref{TheoremMain},
in Section 4 we prove Theorem \ref{TheoremCombinatorics}, and in Section 5 we 
give a few concluding remarks.

\section{Preliminaries}

We begin with expressing the two rank functions as generalized Lambert series.
\begin{proposition}
We have
\begin{align}
	\label{EqRUFinalForm}
	RU(z,q)
	&=
	\frac{1}{(1+z)\aqprod{q,z,z^{-1}}{q}{\infty}}
		\sum_{j=-\infty}^\infty
		\frac{ (1-z^j)(1-z^{j-1})z^{1-j} (-1)^{j}q^{\frac{j(j+3)}{2}}	}
			{(1-z^2q^{j}) (1-z^{-2}q^{j}) }
	,\\
	\label{EqRVFinalForm}
	RV(z,q)
	&=
	\frac{1}{(1+z)\aqprod{q,z,z^{-1}}{q}{\infty}}
		\sum_{j=-\infty}^\infty
		\frac{ (1-z^j)(1-z^{j-1})z^{1-j} (-1)^{j}q^{\frac{j(j+1)}{2}}	}
			{(1-z^2q^{j}) (1-z^{-2}q^{j}) }	
.
\end{align}
\end{proposition}
\begin{proof}
By \cite[Theorem 2.3]{JS5} we have that
\begin{align*}
	&(1+z)\aqprod{z,z^{-1}}{q}{\infty}F(\rho_1,\rho_2,z;q)
	\\	
	&=
	\tfrac{1}{\aqprod{\rho_1,\rho_2,\tfrac{q}{\rho_1\rho_2}}{q}{\infty}}	
	\sum_{j=1}^\infty
		(1-z^j)(1-z^{j-1})z^{1-j}
		\rho_1^{1-j}\rho_2^{1-j}		
		(-1)^{j+1}q^{\frac{j(j-1)}{2}}
		\aqprod{\rho_1,\rho_2}{q}{j-1}
		\aqprod{\rho_1^{-1}q^{j+1},\rho_2^{-1}q^{j+1} }{q}{\infty}	
		\\&\quad\times		
		( 1 - \rho_1^{-1}q^j - \rho_2^{-1}q^j + \rho_1^{-1}q^{3j-1} + \rho_2^{-1}q^{3j-1} - q^{4j-2} )	
	,\\
	&(1+z)\aqprod{z,z^{-1}}{q}{\infty}G(\rho_1,\rho_2,z;q)
	\\	
	&=
	\tfrac{1}{\aqprod{\rho_1,\rho_2,\tfrac{q^2}{\rho_1\rho_2}}{q}{\infty}}	
	\sum_{j=1}^\infty
		(1-z^j)(1-z^{j-1})z^{1-j}
		\rho_1^{1-j}\rho_2^{1-j}		
		(-1)^{j+1}q^{\frac{j(j+1)}{2}-1}
		(1-q^{2j-1})		
		\aqprod{\rho_1,\rho_2}{q}{j-1}
		\\&\qquad\times		
		\aqprod{\rho_1^{-1}q^{j+1},\rho_2^{-1}q^{j+1} }{q}{\infty}	
.
\end{align*}
With $\rho_2=\rho_1^{-1}$ and $\rho_1=z^2$ we find the above reduces to
\begin{align*}
	&(1+z)\aqprod{z,z^{-1}}{q}{\infty}RU(z,q)
	\\	
	&=
	\frac{1}{\aqprod{q}{q}{\infty}}	
	\sum_{j=1}^\infty
	\frac{
		(1-z^j)(1-z^{j-1})z^{1-j}		
		(-1)^{j+1}q^{\frac{j(j-1)}{2}}	
		( 1 - z^2q^j - z^{-2}q^j + z^2q^{3j-1} + z^{-2}q^{3j-1} - q^{4j-2} )
	}{(1-z^2q^{j-1}) (1-z^{-2}q^{j-1}) (1-z^2q^{j}) (1-z^{-2}q^{j}) }
	,\\
	&(1+z)\aqprod{z,z^{-1}}{q}{\infty}RV(z,q)
	\\	
	&=
	\frac{1}{\aqprod{q}{q}{\infty}}	
	\sum_{j=1}^\infty
	\frac{
		(1-z^j)(1-z^{j-1})z^{1-j}		
		(-1)^{j+1}q^{\frac{j(j+1)}{2}-1}	
		(1-q)(1-q^{2j-1})
	}{(1-z^2q^{j-1}) (1-z^{-2}q^{j-1}) (1-z^2q^{j}) (1-z^{-2}q^{j}) }
.
\end{align*}
We note that
\begin{align*}
	\frac{1 - z^2q^j - z^{-2}q^j + z^2q^{3j-1} + z^{-2}q^{3j-1} - q^{4j-2} 	}
		{(1-z^2q^{j-1}) (1-z^{-2}q^{j-1}) (1-z^2q^{j}) (1-z^{-2}q^{j}) }
	&=
	\frac{1}{(1-z^2q^{j-1})(1-z^{-2}q^{j-1})}
	-
	\frac{q^{2j}}{(1-z^2q^{j})(1-z^{-2}q^{j})}
	,\\
	\frac{(1-q)(1-q^{2j-1})}
		{(1-z^2q^{j-1}) (1-z^{-2}q^{j-1}) (1-z^2q^{j}) (1-z^{-2}q^{j}) }
	&=
	\frac{1}{(1-z^2q^{j-1})(1-z^{-2}q^{j-1})}
	-
	\frac{q}{(1-z^2q^{j})(1-z^{-2}q^{j})}
,
\end{align*}
so that
\begin{align*}
	(1+z)\aqprod{z,z^{-1}}{q}{\infty}RU(z,q)
	&=
		\frac{1}{\aqprod{q}{q}{\infty}}	
		\sum_{j=1}^\infty
		\frac{ (1-z^j)(1-z^{j-1})z^{1-j} (-1)^{j+1}q^{\frac{j(j-1)}{2}}	}
			{(1-z^2q^{j-1}) (1-z^{-2}q^{j-1}) }
		\\&\quad		
		-
		\frac{1}{\aqprod{q}{q}{\infty}}	
		\sum_{j=1}^\infty
		\frac{ (1-z^j)(1-z^{j-1})z^{1-j} (-1)^{j+1}q^{\frac{j(j+3)}{2}}	}
			{(1-z^2q^{j}) (1-z^{-2}q^{j}) }
	\\
	&=
		\frac{1}{\aqprod{q}{q}{\infty}}	
		\sum_{j=-\infty}^\infty
		\frac{ (1-z^j)(1-z^{j-1})z^{1-j} (-1)^{j}q^{\frac{j(j+3)}{2}}	}
			{(1-z^2q^{j}) (1-z^{-2}q^{j}) }
	,\\
	(1+z)\aqprod{z,z^{-1}}{q}{\infty}RV(z,q)
	&=
		\frac{1}{\aqprod{q}{q}{\infty}}	
		\sum_{j=1}^\infty
		\frac{ (1-z^j)(1-z^{j-1})z^{1-j} (-1)^{j+1}q^{\frac{j(j+1)}{2}-1}	}
			{(1-z^2q^{j-1}) (1-z^{-2}q^{j-1}) }
		\\&\quad		
		-
		\frac{1}{\aqprod{q}{q}{\infty}}	
		\sum_{j=1}^\infty
		\frac{ (1-z^j)(1-z^{j-1})z^{1-j} (-1)^{j+1}q^{\frac{j(j+1)}{2}}	}
			{(1-z^2q^{j}) (1-z^{-2}q^{j}) }
	\\
	&=
		\frac{1}{\aqprod{q}{q}{\infty}}	
		\sum_{j=-\infty}^\infty
		\frac{ (1-z^j)(1-z^{j-1})z^{1-j} (-1)^{j}q^{\frac{j(j+1)}{2}}	}
			{(1-z^2q^{j}) (1-z^{-2}q^{j}) }
.
\end{align*}
\end{proof}

The proof of Theorem \ref{TheoremMain} will be to replace $n$ by $\ell n+k$,
with $k=0,1,\dots,\ell-1$ in the series in (\ref{EqRUFinalForm}) and 
(\ref{EqRUFinalForm}) and find cancellations between the resulting terms.
For this, we let
\begin{align*}
	T(a,b,\ell) 
	&= 
	\sum_{n=-\infty}^\infty 
		\frac{(-1)^n q^{\frac{\ell^2n(n+1)}{2}+\ell bn} }{(1-q^{\ell^2n+\ell a})}
.
\end{align*}
By letting $n\mapsto -n$ we have the useful fact that
$T(-a,b,\ell)=-q^{\ell a}T(a,-b,\ell)$.
Here and in the proof of Theorem \ref{TheoremMain} we use the notation
\begin{align*}
	E(a) &= \aqprod{q^a}{q^a}{\infty},	
	&P(a) &= \jacprod{q^{\ell a}}{q^{\ell^2}}
,
\end{align*}
where $\ell$ will always be clear from the context.
We note that $P(\ell-a)=P(a)$ and $P(\ell+a)=P(-a)=-q^{-\ell a}P(a)$.
The $r=1$ and $s=2$ case of \cite[Theorem 2.1]{Chan} 
with $q\mapsto q^{\ell^2}$, $a_1\mapsto q^{\ell a}$, $b_1\mapsto q^{\ell b_1}$, and $b_1\mapsto q^{\ell b_2}$  gives
\begin{align*}
	\frac{P(a)E(\ell^2)^2}{P(b_1)P(b_2)}
	&=
	\frac{P(a-b_1)}{P(b_2-b_1)}T(b_1, a-b_2,\ell)	
	+
	\frac{P(a-b_2)}{P(b_1-b_2)}T(b_2, a-b_1,\ell)	
	,
\end{align*}
so that
\begin{align}
	\label{EqChan1}
	T(b_2, a-b_1,\ell)	
	&=
	q^{\ell(b_1-b_2)}
	\frac{P(a-b_1)}{P(a-b_2)}T(b_1, a-b_2,\ell)	
	-
	q^{\ell(b_1-b_2)}
	\frac{P(a)P(b_2-b_1)E(\ell^2)^2} {P(b_1)P(b_2)P(a-b_2)}
	.
\end{align}
By setting $b_1=-b$, $b_2=b$, and using that $T(-b,a-b,\ell)=-q^{\ell b}T(b,b-a,\ell)$, we also deduce that
\begin{align}
	\label{EqChan2}
	T(b, a+b,\ell)	
	&=
	-q^{-\ell b }
	\frac{P(a+b)}{P(a-b)}T(b, b-a,\ell)	
	+
	q^{-\ell b}
	\frac{P(a)P(2b)E(\ell^2)^2} {P(b)^2 P(a-b)}
	.
\end{align}

Additionally we need an identity for $\aqprod{q,\zeta_\ell,\zeta^{-1}_\ell}{q}{\infty}$.
Using the Jacobi triple product identity,
\begin{align*}
	\aqprod{zq,z^{-1},q}{q}{\infty}
	&=
	\sum_{n=-\infty}^\infty
	(-1)^n z^n q^{\frac{n(n+1)}{2}}
,
\end{align*}
we easily deduce that for odd $\ell$ we have
\begin{align}
	\label{EqProdDissection}
	\aqprod{q,\zeta_\ell,\zeta^{-1}_\ell}{q}{\infty}
	&=	
	(1-\zeta_\ell)E(\ell^2)
	\sum_{k=0}^{\frac{\ell-3}{2}}
	(-1)^k(\zeta_\ell^k-\zeta_\ell^{-k-1})q^{\frac{k(k+1)}{2}}P(\tfrac{\ell-1}{2}-k)
	.
\end{align}

\section{Proof of Theorem \ref{TheoremMain}}

\begin{proof}[Proof of (\ref{TheoremMainRU3})]
By (\ref{EqRUFinalForm}) we have
\begin{align*}
	&RU(\zeta_3,q)
	\\
	&=
		\frac{1}{(1+\zeta_3)\aqprod{q,\zeta_3,\zeta_3^{-1}}{q}{\infty}}
		\sum_{j=-\infty}^\infty
		\frac{ (1-\zeta_3^j)(1-\zeta_3^{j-1})\zeta_3^{1-j} (-1)^{j}q^{\frac{j(j+3)}{2}}	}
			{(1-\zeta_3^2q^{j}) (1-\zeta_3^{-2}q^{j}) }
	\\
	&=
		\frac{1}{(1+\zeta_3)(1-\zeta_3)(1-\zeta_3^{-1})\aqprod{q^3}{q^3}{\infty}}
		\sum_{j=-\infty}^\infty
		\frac{ (1-\zeta_3^2)(1-\zeta_3)\zeta_3^{-1} (-1)^{j}q^{\frac{9j^2+21j}{2}+5}	}
			{(1-\zeta_3^2q^{3j+2}) (1-\zeta_3^{-2}q^{3j+2}) }
	\\	
	&=
		\frac{-1}{\aqprod{q^3}{q^3}{\infty}}
		\sum_{j=-\infty}^\infty
		\frac{ (-1)^{j}q^{\frac{9j^2+21j)}{2}+5} (1-q^{3j+2})	}
			{(1-q^{9j+6}) }		
	\\
	&=
		\frac{q^7}{\aqprod{q^3}{q^3}{\infty}}
		\sum_{n=-\infty}^\infty \frac{ (-1)^{j}q^{\frac{9j^2+27j}{2}}}{(1-q^{9j+6}) }	
		-		
		\frac{q^5}{\aqprod{q^3}{q^3}{\infty}}
		\sum_{n=-\infty}^\infty \frac{ (-1)^{j}q^{\frac{9j^2+21j}{2}}}{(1-q^{9j+6}) }	
,
\end{align*}
which is (\ref{TheoremMainRU3}).
\end{proof}
\begin{proof}[Proof of (\ref{TheoremMainRV3})]
By (\ref{EqRVFinalForm}) we have
\begin{align*}
	&RV(\zeta_3,q)
	\\
	&=
		\frac{1}{(1+\zeta_3)\aqprod{q,\zeta_3,\zeta_3^{-1}}{q}{\infty}}
		\sum_{j=-\infty}^\infty
		\frac{ (1-\zeta_3^j)(1-\zeta_3^{j-1})\zeta_3^{1-j} (-1)^{j}q^{\frac{j(j+1)}{2}}	}
			{(1-\zeta_3^2q^{j}) (1-\zeta_3^{-2}q^{j}) }
	\\
	&=
		\frac{1}{(1+\zeta_3)(1-\zeta_3)(1-\zeta_3^{-1})\aqprod{q^3}{q^3}{\infty}}
		\sum_{j=-\infty}^\infty
		\frac{ (1-\zeta_3^2)(1-\zeta_3)\zeta_3^{-1} (-1)^{j}q^{\frac{9j^2+15j}{2}+3}	}
			{(1-\zeta_3^2q^{3j+2}) (1-\zeta_3^{-2}q^{3j+2}) }
	\\	
	&=
		\frac{-1}{\aqprod{q^3}{q^3}{\infty}}
		\sum_{j=-\infty}^\infty
		\frac{ (-1)^{j}q^{\frac{9j^2+15j}{2}+3} (1-q^{3j+2})	}
			{(1-q^{9j+6}) }		
	\\
	&=
		-\frac{q^3}{\aqprod{q^3}{q^3}{\infty}}
		\sum_{n=-\infty}^\infty \frac{ (-1)^{j}q^{\frac{9j^2+15j}{2}}}{(1-q^{9j+6}) }	
		+		
		\frac{q^5}{\aqprod{q^3}{q^3}{\infty}}
		\sum_{n=-\infty}^\infty \frac{ (-1)^{j}q^{\frac{9j^2+21j}{2}}}{(1-q^{9j+6}) }	
,
\end{align*}
which is (\ref{TheoremMainRV3}).
\end{proof}
\begin{proof}[Proof of (\ref{TheoremMainRU5})]
By (\ref{EqRUFinalForm}) we have
\begin{align*}
	&RU(\zeta_5,q)
	\\
	&=
		\frac{1}{(1+\zeta_5)\aqprod{q,\zeta_5,\zeta_5^{-1}}{q}{\infty}}
		\sum_{j=-\infty}^\infty
		\frac{ (1-\zeta_5^j)(1-\zeta_5^{j-1})\zeta_5^{1-j} (-1)^{j}q^{\frac{j(j+3)}{2}}	}
			{(1-\zeta_5^2q^{j}) (1-\zeta_5^{-2}q^{j}) }
	\\
	&=	
		\frac{1}{(1+\zeta_5)\aqprod{q,\zeta_5,\zeta_5^{-1}}{q}{\infty}}
		\sum_{k=2}^4		
		(-1)^k(1-\zeta_5^k)(1-\zeta_5^{k-1})\zeta_5^{1-k} q^{\frac{k(k+3)}{2}}	
		\sum_{j=-\infty}^\infty
		\frac{ (-1)^{j}q^{\frac{25j^2+15j}{2} +5jk}	}
			{(1-\zeta_5^2q^{5j+k}) (1-\zeta_5^{-2}q^{5j+k}) }
	\\
	&=	
		\frac{1}{(1+\zeta_5)\aqprod{q,\zeta_5,\zeta_5^{-1}}{q}{\infty}}
		\sum_{k=2}^4		
		(-1)^k(1-\zeta_5^k)(1-\zeta_5^{k-1})\zeta_5^{1-k} q^{\frac{k(k+3)}{2}}	
		\\&\quad\times
		\sum_{j=-\infty}^\infty		
		\frac{ (-1)^j q^{ \frac{25j(j+1)}{2}-5j+5jk} (1-q^{5j+k})(1-\zeta_5 q^{5j+k})(1-\zeta_5^{-1} q^{5j+k}) }
			{ (1 - q^{25j+5k}) }		
	\\
	&=	
		\frac{1}{(1+\zeta_5)\aqprod{q,\zeta_5,\zeta_5^{-1}}{q}{\infty}}
		\sum_{k=2}^4		
		(-1)^k(1-\zeta_5^k)(1-\zeta_5^{k-1})\zeta_5^{1-k} q^{\frac{k(k+3)}{2}}	
		\\&\quad\times	
		\left(
			T(k,k-1,5) 
			- (1+\zeta_5+\zeta_5^4)q^k T(k,k,5) 
			+ (1+\zeta_5+\zeta_5^4)q^{2k} T(k,k+1,5)	
			- q^{3k} T(k,k+2,5)	
		\right)
.
\end{align*}

In (\ref{EqChan1}) we set $\ell=5$, $a=2+k+c$, $b_1=2$, and $b_2=k$ to get
\begin{align*}
	T(k,k+c,5)
	&=
	q^{10-5k}\frac{P(k+c)}{P(2+c)}T(2,2+c,5)	
	-
	q^{10-5k}\frac{P(2+k+c)P(k-2)E(25)^2}{P(2)P(k)P(2+c)}
	,
\end{align*}
for $k=3,4$ and $c=-1,0,1,2$.
We set $\ell=5$, $a=1$, and $b=2$ in (\ref{EqChan2}) and simplify the products to get
\begin{align*}
	T(2,3,5)
	&=
	q^{-5}\frac{P(2)}{P(1)}T(2,1,5)	
	-
	q^{-5}\frac{P(1)E(25)^2}{P(2)^2}
.
\end{align*}
With these identities we write each of the $T(a,b,5)$ in terms of $T(2,1,5)$
and $T(2,2,5)$ and carefully simplify to find that
\begin{align}
	\label{EqProofRU5Last}
	&
	(1+\zeta_5)\aqprod{q,\zeta_5,\zeta_5^{-1}}{q}{\infty}RU(\zeta_5,q)
	\nonumber\\
	&=
		T(2, 1, 5)\left( - (2+2\zeta_5 + \zeta_5^3)q^4\tfrac{P(2)}{P(1)} + (1+\zeta_5-2\zeta_5^3)q^5  \right)	
		+					
		T(2, 2, 5)\left(-(2+2\zeta_5+\zeta_5^3 )q^7 + (1 + \zeta_5 - 2\zeta_5^3)q^8\tfrac{P(1)}{P(2)} \right)	
		\nonumber\\&\quad
		+
		(2+2\zeta_5+\zeta_5^3) q\tfrac{P(2)E(25)^2}{P(1)^2}		
		-
		(1 + \zeta_5 - 2\zeta_5^3 )q^2 \tfrac{E(25)^2}{P(1)}
	\nonumber\\
	&=
	(1+\zeta_5)\aqprod{q,\zeta_5,\zeta_5^{-1}}{q}{\infty}
	\left(
		q\tfrac{E(25)}{P(1)^2}							
		-
		\tfrac{q^7}{E(25)P(2)}T(2, 2, 5)
		-
		\tfrac{q^4}{E(25)P(1)}T(2, 1, 5)  
	\right)
,
\end{align}
where the last equality follows from using (\ref{EqProdDissection}) to get that
\begin{align*}
	(1+\zeta_5)\aqprod{q,\zeta_5,\zeta_5^{-1}}{q}{\infty}
	&=
	E(25)(  (2+2\zeta_5+\zeta_5^3)P(2) - (1 + \zeta_5 - 2\zeta_5^3)qP(1)  )	
	.
\end{align*}
We see (\ref{EqProofRU5Last}) now immediately implies (\ref{TheoremMainRU5}).

\end{proof}
\begin{proof}[Proof of (\ref{TheoremMainRV5})]
By (\ref{EqRVFinalForm}) we have
\begin{align*}
	&RV(\zeta_5,q)
	\\
	&=
		\frac{1}{(1+\zeta_5)\aqprod{q,\zeta_5,\zeta_5^{-1}}{q}{\infty}}
		\sum_{j=-\infty}^\infty
		\frac{ (1-\zeta_5^j)(1-\zeta_5^{j-1})\zeta_5^{1-j} (-1)^{j}q^{\frac{j(j+1)}{2}}	}
			{(1-\zeta_5^2q^{j}) (1-\zeta_5^{-2}q^{j}) }
	\\
	&=	
		\frac{1}{(1+\zeta_5)\aqprod{q,\zeta_5,\zeta_5^{-1}}{q}{\infty}}
		\sum_{k=2}^4		
		(-1)^k(1-\zeta_5^k)(1-\zeta_5^{k-1})\zeta_5^{1-k} q^{\frac{k(k+1)}{2}}	
		\sum_{j=-\infty}^\infty
		\frac{ (-1)^{j}q^{\frac{25j^2+5j}{2} +5jk}	}
			{(1-\zeta_5^2q^{5j+k}) (1-\zeta_5^{-2}q^{5j+k}) }
	\\
	&=	
		\frac{1}{(1+\zeta_5)\aqprod{q,\zeta_5,\zeta_5^{-1}}{q}{\infty}}
		\sum_{k=2}^4		
		(-1)^k(1-\zeta_5^k)(1-\zeta_5^{k-1})\zeta_5^{1-k} q^{\frac{k(k+1)}{2}}	
		\\&\quad\times
		\sum_{j=-\infty}^\infty		
		\frac{ (-1)^j q^{ \frac{25j(j+1)}{2}-10j+5jk} (1-q^{5j+k})(1-\zeta_5 q^{5j+k})(1-\zeta_5^{-1} q^{5j+k}) }
			{ (1 - q^{25j+5k}) }		
	\\
	&=	
		\frac{1}{(1+\zeta_5)\aqprod{q,\zeta_5,\zeta_5^{-1}}{q}{\infty}}
		\sum_{k=2}^4		
		(-1)^k(1-\zeta_5^k)(1-\zeta_5^{k-1})\zeta_5^{1-k} q^{\frac{k(k+1)}{2}}	
		\\&\quad\times	
		\left(
			T(k,k-2,5) 
			- (1+\zeta_5+\zeta_5^4)q^k T(k,k-1,5) 
			+ (1+\zeta_5+\zeta_5^4)q^{2k} T(k,k,5)	
			- q^{3k} T(k,k+1,5)	
		\right)
.
\end{align*}

In (\ref{EqChan1}) we set $\ell=5$, $a=3+k+c$, $b_1=3$, and $b_2=k$ to get
\begin{align*}
	T(k,k+c,5)
	&=
	q^{15-5k}\frac{P(k+c)}{P(3+c)}T(3,3+c,5)	
	-
	q^{15-5k}\frac{P(3+k+c)P(k-3)E(25)^2}{P(3)P(k)P(3+c)}
	,
\end{align*}
for $k=2,4$ and $c=-2,-1,0,1$.
We set $\ell=5$, $a=1$, and $b=3$ in (\ref{EqChan2}) and simplify the products to get
\begin{align*}
	T(3,4,5)
	&=
	q^{-5}\frac{P(1)}{P(2)}T(3,2,5)	
	+
	q^{-10}\frac{P(1)^2E(25)^2}{P(2)^3}
.
\end{align*}
With these identities we write each of the $T(a,b,5)$ in terms of $T(3,1,5)$
and $T(3,3,5)$ and carefully simplify to find that
\begin{align}
	\label{EqProofRV5Last}
	&
	(1+\zeta_5)\aqprod{q,\zeta_5,\zeta_5^{-1}}{q}{\infty}RV(\zeta_5,q)
	\nonumber\\
	&=
		T(3, 1, 5)
		\left(
			(-2-2\zeta_5-\zeta_5^3)q^5\tfrac{P(2)}{P(1)}
			+
			(1+\zeta_5-2\zeta_5^3)q^6
		\right)
		+
		T(3, 3, 5)
		\Big(
			(2+2\zeta_5+\zeta_5^3)q^{12}
			\nonumber\\&\quad
			+			
			(-1-\zeta_5+2\zeta_5^3)q^{13}\tfrac{P(1)}{P(2)}
		\Big)
		+
		(2+2\zeta_5+\zeta_5^3)q^2\tfrac{E(25)^2}{P(1)}
		+
		(-3-3\zeta_5+\zeta_5^3)q^3\tfrac{E(25)^2}{P(2)}
		\nonumber\\&\quad		
		+
		(1+\zeta_5-2\zeta_5^3)q^4\tfrac{P(1)E(25)^2}{P(2)^2}
	\nonumber\\
	&=
	(1+\zeta_5)\aqprod{q,\zeta_5,\zeta_5^{-1}}{q}{\infty}
	\left(
		-\tfrac{q^5}{P(1)E(49)}	T(3, 1, 5)
		+
		\tfrac{q^{12}}{P(2)E(49)}
		T(3, 3, 5)
		+
		q^2\tfrac{E(25)}{P(1)P(2)}		
		-
		q^3\tfrac{E(25)}{P(2)^2}
	\right)
,
\end{align}
where the last equality also follows from
\begin{align*}
	(1+\zeta_5)\aqprod{q,\zeta_5,\zeta_5^{-1}}{q}{\infty}
	&=
	E(25)(  (2+2\zeta_5+\zeta_5^3)P(2) - (1 + \zeta_5 - 2\zeta_5^3)qP(1)  )	
	.
\end{align*}
We see (\ref{EqProofRV5Last}) now immediately implies (\ref{TheoremMainRV5}).

\end{proof}
\begin{proof}[Proof of (\ref{TheoremMainRU7}).]
By (\ref{EqRUFinalForm}) we have that
\begin{align*}
	&RU(\zeta_7,q)
	\\
	&=
		\frac{1}{(1+\zeta_7)\aqprod{q,\zeta_7,\zeta_7^{-1}}{q}{\infty}}
		\sum_{j=-\infty}^\infty
		\frac{ (1-\zeta_7^j)(1-\zeta_7^{j-1})\zeta_7^{1-j} (-1)^{j}q^{\frac{j(j+3)}{2}}	}
			{(1-\zeta_7^2q^{j}) (1-\zeta_7^{-2}q^{j}) }
	\\
	&=	
		\frac{1}{(1+\zeta_7)\aqprod{q,\zeta_7,\zeta_7^{-1}}{q}{\infty}}
		\sum_{k=2}^6		
		(-1)^k(1-\zeta_7^k)(1-\zeta_7^{k-1})\zeta_7^{1-k} q^{\frac{k(k+3)}{2}}	
		\\&\quad\times		
		\sum_{j=-\infty}^\infty
		\frac{ (-1)^{j}q^{\frac{49j^2+21j}{2} +7jk}	}
			{(1-q^{49j+7k})  }
        (1-q^{7j+k})(1-\zeta_7q^{7j+k})(1-\zeta_7^3q^{7j+k})(1-\zeta_7^4q^{7j+k})(1-\zeta_7^6q^{7j+k})		
	\\
	&=	
		\frac{1}{(1+\zeta_7)\aqprod{q,\zeta_7,\zeta_7^{-1}}{q}{\infty}}
		\sum_{k=2}^6		
		(-1)^k(1-\zeta_7^k)(1-\zeta_7^{k-1})\zeta_7^{1-k} q^{\frac{k(k+3)}{2}}	
		\left(
			T(k,k-2,5) 
			+ (\zeta_7^2+\zeta_7^5)q^k T(k,k-1,7) 
			\right.\\&\qquad			
			+ (1+\zeta_7^3+\zeta_7^4)q^{2k} T(k,k,7)	
			- (1+\zeta_7^3+\zeta_7^4)q^{3k} T(k,k+1,7)	
			- (\zeta_7^2+\zeta_7^5)q^{4k} T(k,k+2,7)
			\\&\left.\qquad			
			- q^{5k} T(k,k+3,7)
		\right)
.
\end{align*}

In (\ref{EqChan1}) we set $\ell=7$, $a=3+k+c$, $b_1=3$, and $b_2=k$ to get
\begin{align*}
	T(k,k+c,7)
	&=
	q^{21-7k}\frac{P(k+c)}{P(3+c)}T(3,3+c,7)	
	-
	q^{21-7k}\frac{P(3+k+c)P(k-3)E(49)^2}{P(3)P(k)P(3+c)}
	,
\end{align*}
for $k=2,4,5,6$ and $c=-2,\dots,3$.
We set $\ell=7$, $a=1,2$, and $b=3$ in (\ref{EqChan2}) and simplify the products to get
\begin{align*}
	T(3,4,7)
	&=
		q^{-7}\frac{P(3)}{P(2)}T(3,2,7)	
		-
		q^{-7}\frac{P(1)^2 E(49)^2}{P(3)^2 P(2)}
	,\\
	T(3,5,7)
	&=
		q^{-14}\frac{P(2)}{P(1)}T(3,1,7)	
		-
		q^{-14}\frac{P(2) E(49)^2}{P(3)^2}
.
\end{align*}
With these identities we write each of the $T(a,b,7)$ in terms of $T(3,1,7)$,
$T(3,2,7)$,
and $T(3,3,7)$ and carefully simplify to find that
\begin{align}
	\label{EqProofRU7Last1}
	&(1+\zeta_7)\aqprod{q,\zeta_7,\zeta_7^{-1}}{q}{\infty}RU(\zeta_7,q)	
	\nonumber\\
	&=
		T(3, 1, 7)\Big(
			(-3-3\zeta_7-2\zeta_7^3-4\zeta_7^4-2\zeta_7^5)q^6\tfrac{P(3)}{P(1)}
			+
			(1+\zeta_7+\zeta_7^3+3\zeta_7^4+\zeta_7^5)q^7\tfrac{P(2)}{P(1)}
			\nonumber\\&\quad			
			+
			(4+4\zeta_7+\zeta_7^3+4\zeta_7^4+\zeta_7^5)q^9
		\big)
		+	
		T(3, 2, 7)
		\Big(
			(-1-\zeta_7-2\zeta_7^3-\zeta_7^4-2\zeta_7^5)q^{11}\tfrac{P(3)}{P(2)}
			+
			(-1-\zeta_7+2\zeta_7^4)q^{12}
			\nonumber\\&\quad			
			+
			(1+\zeta_7-\zeta_7^3-\zeta_7^5)q^{14}\tfrac{P(1)}{P(2)}
		\Big)
		+
		T(3, 3, 7)
		\Big(
			(2+2\zeta_7+3\zeta_7^4)q^{15}
			+
			(-2-2\zeta_7-\zeta_7^3-\zeta_7^4-\zeta_7^5)q^{16}\tfrac{P(2)}{P(3)}
			\nonumber\\&\quad			
			+
			(-3-3\zeta_7-2\zeta_7^3-4\zeta_7^4-2\zeta_7^5)q^{18}\tfrac{P(1)}{P(3)}
		\Big)
		+
		(2+2\zeta_7+\zeta_7^3+\zeta_7^4+\zeta_7^5)q\tfrac{P(2)E(49)^2}{P(1)^2}
		\nonumber\\&\quad
		-
		(2+2\zeta_7+3\zeta_7^4)q^2\tfrac{P(3)E(49)^2}{P(2)P(1)}
		+
		(2+2\zeta_7+3\zeta_7^3+4\zeta_7^4+3\zeta_7^5)q^3\tfrac{E(49)^2}{P(1)}
		\nonumber\\&\quad
		-
		(3+3\zeta_7+2\zeta_7^3+4\zeta_7^4+2\zeta_7^5)q^4\tfrac{P(2)E(49)^2}{P(3)P(1)}
		+
		(2+2\zeta_7+3\zeta_7^4)q^5\tfrac{E(49)^2}{P(2)}		
		+
		(2+2\zeta_7+\zeta_7^3+\zeta_7^4+\zeta_7^5)q^6\tfrac{E(49)^2}{P(3)}
		\nonumber\\&\quad
		-
		(2+2\zeta_7+2\zeta_7^3+6\zeta_7^4+2\zeta_7^5)q^7\tfrac{P(2)E(49)^2}{P(3)^2}
		+
		(2+2\zeta_7+3\zeta_7^4)q^{7}\tfrac{P(1)E(49)^2}{P(2)^2}
		\nonumber\\&\quad
		-
		(2+2\zeta_7+\zeta_7^3+\zeta_7^4+\zeta_7^5)q^{8}\tfrac{P(1)E(49)^2}{P(3)P(2)}		
		-
		(4+4\zeta_7 +\zeta_7^3+4\zeta_7^4+\zeta_7^5)q^9\tfrac{P(1)E(49)^2}{P(3)^2}
		\nonumber\\&\quad
		+
		(2 +2\zeta_7 +3\zeta_7^3 + 4\zeta_7^4 +3\zeta_7^5)q^{11}\tfrac{P(1)^2 E(49)^2}{P(3)^2 P(2)}
		+
		(1+\zeta_7+\zeta_7^3+3\zeta_7^4+\zeta_7^5)q^{14}\frac{P(1)^3 E(49)^2}{P(3)^3P(2)}
.
\end{align}
We slightly alter (\ref{EqProofRU7Last1}) before proceeding. By Lemma 4 of 
\cite{AS} with $b=3$, $c=2$, and $d=1$,
we know
\begin{align*}
	P(3)^3P(1) - P(2)^3P(3) + q^7P(1)^3P(2) = 0 
,
\end{align*}
which yields
\begin{align*}
	q\tfrac{P(2)}{P(1)^2}-q^8\tfrac{P(1)}{P(2) P(3)}
	&=			
		q\tfrac{P(3)^2}{P(1) P(2)^2}
	,\\
	q^{11}\tfrac{P(1)^2}{P(2) P(3)^2}		
	&=
		q^4\tfrac{P(2)}{P(1) P(3)} - q^4\tfrac{P(3)}{P(2)^2}		
	,\\
	q^{14}\tfrac{P(1)^3}{P(2) P(3)^3}
	&=
		-q^7\tfrac{P(1)}{P(2)^2} + q^7\tfrac{P(2)}{P(3)^2}			
.
\end{align*}
These identities allow us to rewrite (\ref{EqProofRU7Last1}) as
\begin{align}
	\label{EqProofRU7Last2}
	&(1+\zeta_7)\aqprod{q,\zeta_7,\zeta_7^{-1}}{q}{\infty}RU(\zeta_7,q)	
	=
	\nonumber\\
	&=
		T(3, 1, 7)\Big(
			(-3-3\zeta_7-2\zeta_7^3-4\zeta_7^4-2\zeta_7^5)q^6\tfrac{P(3)}{P(1)}
			+
			(1+\zeta_7+\zeta_7^3+3\zeta_7^4+\zeta_7^5)q^7\tfrac{P(2)}{P(1)}
			\nonumber\\&\quad			
			+
			(4+4\zeta_7+\zeta_7^3+4\zeta_7^4+\zeta_7^5)q^9
		\big)
		+	
		T(3, 2, 7)
		\Big(
			(-1-\zeta_7-2\zeta_7^3-\zeta_7^4-2\zeta_7^5)q^{11}\tfrac{P(3)}{P(2)}
			+
			(-1-\zeta_7+2\zeta_7^4)q^{12}
			\nonumber\\&\quad			
			+
			(1+\zeta_7-\zeta_7^3-\zeta_7^5)q^{14}\tfrac{P(1)}{P(2)}
		\Big)
		+
		T(3, 3, 7)
		\Big(
			(2+2\zeta_7+3\zeta_7^4)q^{15}
			-
			(2+2\zeta_7+\zeta_7^3+\zeta_7^4+\zeta_7^5)q^{16}\tfrac{P(2)}{P(3)}
			\nonumber\\&\quad			
			-
			(3+3\zeta_7+2\zeta_7^3+4\zeta_7^4+2\zeta_7^5)q^{18}\tfrac{P(1)}{P(3)}
		\Big)
		+
		(2+2\zeta_7+\zeta_7^3+\zeta_7^4+\zeta_7^5)q\tfrac{P(3)^2E(49)^2}{P(1) P(2)^2}
		\nonumber\\&\quad
		-
		(2+2\zeta_7+3\zeta_7^4)q^2\tfrac{P(3)E(49)^2}{P(2)P(1)}
		+
		(2+2\zeta_7+3\zeta_7^3+4\zeta_7^4+3\zeta_7^5)q^3\tfrac{E(49)^2}{P(1)}
		+
		(-1-\zeta_7+\zeta_7^3+\zeta_7^5)q^4\tfrac{P(2)E(49)^2}{P(3)P(1)}
		\nonumber\\&\quad
		-
		(2 +2\zeta_7 +3\zeta_7^3 + 4\zeta_7^4 +3\zeta_7^5)q^{4}\tfrac{P(3) E(49)^2}{P(2)^2}
		+
		(2+2\zeta_7+3\zeta_7^4)q^5\tfrac{E(49)^2}{P(2)}		
		+
		(2+2\zeta_7+\zeta_7^3+\zeta_7^4+\zeta_7^5)q^6\tfrac{E(49)^2}{P(3)}
		\nonumber\\&\quad
		+
		(1+\zeta_7-\zeta_7^3-\zeta_7^5))q^{7}\tfrac{P(1)E(49)^2}{P(2)^2}
		-
		(1+\zeta_7+\zeta_7^3+3\zeta_7^4+\zeta_7^5))q^7\tfrac{P(2)E(49)^2}{P(3)^2}
		\nonumber\\&\quad
		-
		(4+4\zeta_7 +\zeta_7^3+4\zeta_7^4+\zeta_7^5)q^9\tfrac{P(1)E(49)^2}{P(3)^2}
	\nonumber\\
	&=
	(1+\zeta_7)\aqprod{q,\zeta_7,\zeta_7^{-1}}{q}{\infty}
	\left(
		-(\zeta_7^2+\zeta_7^5)\tfrac{q^{15}}{P(3)E(49)}T(3,3,7)
		+
		(\zeta_7+\zeta_7^6)\tfrac{q^{11}}{P(2)E(49)}T(3,2,7)
		\right.\nonumber\\&\quad
		-
		(1+\zeta_7^3+\zeta_7^4)\tfrac{q^6 }{P(1)E(49)}T(3,1,7)
  		+
  		q\tfrac{E(49)P(3)}{P(2)^2P(1)}
		-
		(\zeta_7^3+\zeta^4)q^2 \frac{E(49)}{P(1)P(2)}
		+q^3 \tfrac{E(49)}{P(1)P(3)}
		\nonumber\\&\quad\left.
		+
		(\zeta_7+\zeta_7^6)q^4\tfrac{E(49)}{P(2)^2}
		+
		(1+\zeta_7^3+\zeta_7^4)q^6\tfrac{E(49)}{P(3)^2}
	\right)
,
\end{align}
where the last equality follows from using (\ref{EqProdDissection}) to get 
that
\begin{align*}
	(1+\zeta_7)\aqprod{q,\zeta_7,\zeta_7^{-1}}{q}{\infty}
	&=
	E(49)( (2+2\zeta_7+\zeta_7^3+\zeta_7^4+z^5)P(3)
		+(-1-\zeta_7+\zeta_7^3+\zeta_7^5)qP(2)
		\\&\quad
		-(1+\zeta_7+\zeta_7^3+3\zeta_7^4+\zeta_7^5)q^3P(1))
	.
\end{align*}
We see (\ref{EqProofRU7Last2}) immediately implies (\ref{TheoremMainRU7}).

\end{proof}

\section{Proof of Theorem \ref{TheoremCombinatorics}}

Our proof is similar to the rearrangements and interpretations used
by Andrews and Garvan in \cite{AndrewsGarvan} to interpret the vector crank
of \cite{Garvan1} 
as the crank of ordinary partitions.
We need only the $q$-binomial theorem in the form of
\begin{align*}
	\frac{\aqprod{tw}{q}{\infty}}{\aqprod{t,w}{q}{\infty}}
	&=
	\sum_{m=0}^\infty	
	\frac{t^m}{\aqprod{wq^m}{q}{\infty}\aqprod{q}{q}{m}}	
.
\end{align*}
With $t=z^{-1}q^n$ and $w=zq^{n+1}$ this gives us that
\begin{align*}
	RU(z,q)
	&=
		\sum_{n=1}^\infty
		\frac{q^n\aqprod{q^{2n+1}}{q}{\infty}}
			{\aqprod{zq^n,z^{-1}q^n,z^2q^n,z^{-2}q^n}{q}{\infty}}
	\\
	&=	
		\sum_{n=1}^\infty
		\frac{q^n}{(1-zq^n)\aqprod{z^2q^n,z^{-2}q^n}{q}{\infty}}	
		\sum_{m=0}^\infty
		\frac{z^{-m}q^{nm}}{\aqprod{zq^{n+m+1}}{q}{\infty}\aqprod{q}{q}{m}}
	\\
	&=	
		\sum_{n=1}^\infty
		\frac{q^n}{\aqprod{zq^n,z^2q^n,z^{-2}q^n}{q}{\infty}}	
		+	
		\sum_{n=1}^\infty
		\sum_{m=1}^\infty		
		\frac{z^{-m}q^{nm+n}}{(1-zq^n)\aqprod{zq^{n+m+1},z^2q^n,z^{-2}q^n}{q}{\infty}\aqprod{q}{q}{m}}	
	\\
	&=
		\sum_{n=1}^\infty
		\frac{q^n}{\aqprod{zq^n,z^2q^n,z^{-2}q^n}{q}{\infty}}	
		\\&\quad		
		+	
		\sum_{n=1}^\infty
		\sum_{m=1}^\infty		
		\frac{q^{n}}{(1-zq^n)\aqprod{q^{n+1}}{q}{m}\aqprod{zq^{n+m+1},z^2q^n,z^{-2}q^n}{q}{\infty}}		
		\frac{z^{-m}q^{nm}\aqprod{q}{q}{n+m}}{\aqprod{q}{q}{n}\aqprod{q}{q}{m}}
	.
\end{align*}
Loosely speaking, $z$ counts certain parts in $\pi_1$, 
$z^2$ counts $\#(\pi_2)$, $z^{-2}$ counts $\#(\pi_3)$, and
$z^{-1}$ counts $\#(\pi_4)$.
We see the first sum is the generating function for the partition quadruples 
$(\pi_1,\pi_2,\pi_3,\pi_4)$ from $U$ when $\pi_4$ is the empty partition,
where the power of $q$ is the sum of the parts of the $\pi_i$
and the power of $z$ is $\#(\pi_1)-1+2\#(\pi_2)-2\#(\pi_3)$. That is to say the
first sum is the generating function for the number of partition quadruples of
$n$ from $U$ with $u$-rank equal to $m$ when $\pi_4$ is the empty partition.
For the second sum
we first note that
$\frac{\aqprod{q}{q}{n+m}}{\aqprod{q}{q}{n}\aqprod{q}{q}{m}}$ is well
known to be the generating function for partitions with
parts at most $n$ in size and at most $m$ parts in total \cite[Theorem 3.1]{AndrewsBook}, so
that $\frac{q^{nm}\aqprod{q}{q}{n+m}}{\aqprod{q}{q}{n}\aqprod{q}{q}{m}}$
is the generating function for partitions into exactly $m$ parts
and with all parts between $n$ and $2n$. We see then the second sum
is the generating function for the partition quadruples
$(\pi_1,\pi_2,\pi_3,\pi_4)$ from $U$ when $\pi_4$ is non-empty,
where the power of $q$ is the sum of the parts of the $\pi_i$
and the power of $z$ is $\omega(\pi_1,\pi_2,\pi_3,\pi_4)-1+2\#(\pi_2)-2\#(\pi_3)-\#(\pi_4)$. That 
is to say the
second sum is the generating function for the number of partition quadruples of
$n$ from $U$ with $u$-rank equal to $m$ when $\pi_4$ is non-empty.
Thus $ru(m,n)$ is the number of partition quadruples of $n$ from $U$
with $u$-rank equal to $m$.

In the same fashion we find $rv(m,n)$ to be the number of partition quadruples of 
$n$ from $V$ with $v$-rank equal to $m$. The remainder of Theorem 
\ref{TheoremCombinatorics} follows from the fact that
$ru(k,\ell,\ell n+a)=\frac{u(\ell n+a)}{\ell}$ for $(\ell,a)=(3,0)$, $(5,0)$, 
$(5,3)$, $(7,0)$, and $(7,5$)
and 
$rv(k,\ell,\ell n+a)=\frac{v(\ell n+a)}{\ell}$ for $(\ell,a)=(3,1)$, $(5,1)$, 
and $(5,4)$
for all $k$
as established by Theorem \ref{TheoremMain}.

\section{Remarks}
It is somewhat surprising that the dissections of $RU(\zeta_\ell,q)$ and
$RV(\zeta_\ell,q)$ are easier to handle than the modulo $\ell$ dissections
of $U(q)$ and $V(q)$. 
However, the formulas for $U(q)$ and $V(q)$ modulo $\ell$ are still
necessary, as $RU(z,q)$ and $RV(z,q)$ do not also explain the modulo $13$ congruences.
In particular one can check that the coefficient
of $q^{13}$ in $RU(\zeta_{13},q)$ is non-zero. Additionally one can check that 
the coefficient of $q^{13}$ is non-zero in
$F(\zeta_{13}^a,\zeta_{13}^b,\zeta_{13}^c;q)$ 
for all choices of $a$, $b$, and $c$. We leave it as an open problem to find a
statistic to explain the modulo $13$ congruences for $u(n)$ and $v(n)$.

\bibliographystyle{abbrv}
\bibliography{RanksForTwoPartitionQuadruplesFunctionsRef}

\end{document}